\documentclass{amsart}
\usepackage{amssymb,amsfonts,amsmath}
\usepackage[all]{xy}
\usepackage[shortlabels]{enumitem}
\usepackage{hyperref, color}		
\newtheorem{theorem}{Theorem}[section]
\newtheorem{maintheorem}{Theorem}

\newtheorem{corollary}{Corollary}[section]
\theoremstyle{definition}

\theoremstyle{lemma}
\newtheorem{lemma}{Lemma}[section]

\theoremstyle{remark}
\newtheorem{remark}{Remark}
\theoremstyle{claim}

\theoremstyle{theoremA}

\theoremstyle{theoremB}

\newcommand{\real}{\mathbb{R}}

\newcommand{\la}{\lambda}

\newcommand{\Om}{\Omega}
\newcommand{\na}{\nabla}

\newcommand{\Ga}{\Gamma}

\newcommand{\De}{\Delta}

\newcommand{\vol}{\text{\rm vol}}

\newcommand{\p}{\partial}

\begin{document}


\title[The $p$-Hyperbolicity of Infinity Volume Ends]
{The $p$-Hyperbolicity of Infinity Volume Ends \\ and Applications
}
\author{M. Batista}
\address{Instituto de Matem\'atica, Universidade Fe\-deral de Alagoas, Macei\'o, AL, CEP 57072-970, Brazil}\email{mhbs@mat.ufal.br}
\author{M. P. Cavalcante}
\address{Instituto de Matem\'atica, Universidade Fe\-deral de Alagoas, Macei\'o, AL, CEP 57072-970, Brazil}
\email{marcos@pos.mat.ufal.br}
\thanks{The second author was  partially supported by CNPq}

\author{N. L. Santos}
\address{Departamento de Matem‡tica, Universidade Fe\-deral do Piau\'i, Teresina, PI, CEP 64049-550, Brazil}
\email{newtonls@ufpi.edu.br}
\date{\today}
\subjclass[2010]{Primary 31C12, 58C40; Secondary 47J10, 53C26}

\maketitle

\thispagestyle{empty}
\begin{abstract} In this paper we prove a characterization of $p$-hyperbolic ends
on complete Riemannian manifolds which carries a Sobolev type inequality. 
\end{abstract}

\section{Introduction}

Let $M^n$ be a complete noncompact Riemannian manifold.
Given $p\geq1$,  we recall that the $p$-Laplacian operator on $M$ is defined by
$$
\De_p u  := \textrm{div} (|\nabla u |^{p-2}\nabla u),
$$
for $u\in W^{1,p}_{loc}(M)$. It is  the Euler-Lagrange operator
associated to the $p$-energy functional,
$
E_p(u) := \int_M |\nabla u|^p\, \textrm{d}M.
$
This non-linear operator appears naturally in many situations, and we refer the
reader to \cite{GP}, \cite{Hu}, \cite{Ly} and the references cited therein for further information.
As usual, we say that a function $u$ 
is $p$-harmonic if $\De_p u =0.$

\medskip
Let $E\subset M$ be an \emph{end} of $M$, that is an unbounded connect component
of $M\setminus \Om$, for some compact subset, $\Om\subset M$, with smooth boundary.
We say that $E$ is $p$-\emph{parabolic}
(see Definition 2.4 of \cite{Li08} for $p=2$ and Theorem 2.5 of \cite{PST} for
the general case)  if it does not admit a $p$-harmonic function, $f:E\to \real$, satisfying:
$$
\begin{cases}
f|_{\p E}=1; \\
\liminf_{{y\to \infty}\atop{y\in E}} f(y)<1.
\end{cases}
$$
Otherwise, we say that $E$ is a $p$-\emph{hyperbolic} end of $M$.

\medskip

In \cite{LW02} Li and Wang obtained the following characterization of
the ends of complete manifolds.
For simplicity, we omit the volume element of integrals.

\begin{maintheorem}[Corollary 4 of \cite{LW02}] Let $E$ be an end
of a complete manifold. Suppose that, for some constants $\nu \geq 1$ and $C>0$, $E$
satisfies a Sobolev-type inequality of the form
\begin{equation}\label{sob-nu}
\Big(\int_E |u|^{2\nu}\Big)^{\frac{1}{\nu}} \leq C\int_E |\na u|^2,
\end{equation}
for all compactly supported Sobolev function $u \in  W_0^{1,2}(E)$.
Then $E$ must either have finite volume or be $2$-hyperbolic.
\end{maintheorem}

In our first result, we extend the above theorem for $p$-hyperbolic ends. Namely

\begin{theorem}\label{ends}
Let $E$ be an end of a complete Riemannian manifold.  Assume that
for some constants, $1<p\leq q<\infty$ and $C>0$, $E$ satisfies a Sobolev-type
inequality of the form
\begin{equation}\label{sobolev}
\left( \int_E |u|^{q}\right)^\frac{p}{q}\leq C\int_E |\nabla u|^p,
\end{equation}
for all $u\in W^{1,p}_0(E)$. Then $E$ must either have
finite  volume or be $p$-hyperbolic.
\end{theorem}

To prove this theorem we apply the techniques developed in \cite{LW02} and 
a lemma due to Cacciopolli (see Lemma \ref{cacci} in Section \ref{pre}).
Some application for Cheng's type inequalities are given in the Section \ref{applications}

\medskip
Our next result is characterization of $p$-hyperbolic ends in the context of
submanifolds as recently obtained in \cite{CMV}. Bellow, let us denote by
$H$ the mean curvature vector field of an isometric immersion $x:M^m\to \bar M$ 
and by $||H||_{L^q(E)}$ its Lebesgue $L^q$-norm on $E\subset M$.

\begin{theorem}\label{nonparabolic} Let $x:M^m\to \bar M$, with $m\geq 3$, be an 
isometric immersion of a complete non-compact manifold $M$ in a  manifold $\bar M$ 
with nonpositive  radial curvature.  Given, $1<p<m$,
let $E$ be an end of $M$ such that the mean curvature vector satisfies $\|H\|_{L^q(E)}<\infty$, 
for some  $q\in [p,m]. $
Then $E$ must either have finite volume or be $p$-hyperbolic.  
\end{theorem}

As a direct consequence, we have:
\begin{corollary}
Let $x:M^m\to \bar M$, with $m\geq 3$, be a minimal isometric immersion of a complete 
manifold $M$ in a manifold $\bar M$ with nonpositive radial curvature. Then, each end of $M$ is
$p$-hyperbolic, for each $p\in (1,m)$.

\end{corollary}

The main tool in the proof of Theorem \ref{nonparabolic} is the Hofmann-Spruck inequality 
\cite{HS} and its refinement given in \cite{BM}.


\section{Preliminaries on $p$-Harmonic function}\label{pre}
In this section we prove two basic results which will be used to prove Theorems \ref{ends}
and \ref{nonparabolic} as well for Cheng's inequalities in Section \ref{applications}.
We first refine a technical lemma due to Caccioppoli (see Lemma 2.9 of \cite{Lind}).

\begin{lemma}(Caccioppoli) \label{cacci}Let $\Om\subset M$ be a compact set and let
$\Ga$ be a connect component of $\partial \Om$.
Given $p>1$, if $u$ is a weak solution for the $p$-Laplace equation in $\Om$ such
that $u$ vanishes on $\Ga$, then 
$$
\int_\Om \varphi^p|\nabla u|^p \le p^p \int_\Om u^p |\nabla \varphi|^p,
$$
for all smooth function $\varphi$ such that $0\le \varphi \le 1$ and $\varphi$ equals
zero in $\partial \Om \setminus \Ga$.
\end{lemma}

\begin{proof}
Since $\De_p u=0$ weakly in $\Om$ and $\varphi^pu$ vanishes on $\partial \Om$
we have
$$
\int_\Om\langle \nabla (\varphi^p u), |\nabla u|^{p-2}\nabla u\rangle=0.
$$
Thus, using H\"older inequality,
\begin{eqnarray*}
\int_\Om \varphi^p |\nabla u|^p&=&
-p\int_\Om\varphi^{p-1}u\,\langle  |\nabla u|^{p-2}\nabla u, \nabla\varphi\rangle\\
&\le& p\int_\Om |\varphi \nabla u|^{p-1} |u\nabla \varphi| \\
&\le& p\bigg(\int_\Om \varphi^p |\nabla u|^{p}\bigg)^{(p-1)/p}
\bigg(\int_\Om |u|^p |\nabla \varphi|^{p}\bigg)^{1/p}.
\end{eqnarray*}
This completes the proof of the lemma.
\end{proof}

The next lemma is a well known result for the Laplacian operator and the proof follows
closely the one in \cite{Li}. We include the proof here for the sake of completeness.

\begin{lemma}\label{poli}
Let $M$ be a complete noncompact Riemannian manifold.
If $M$ has a polynomial volume growth, then $\lambda_{1,p}(M)=0.$
\end{lemma}
\begin{proof}
By hypothesis, there exist $C>0$ and $k\geq 0$ such that
$$
V(r):=Vol(B_r)\leq C\, r^k,
$$
for all $r>0$ big enough.
On the other hand, from the variational characterization of $\lambda_{1,p}(M)$ we have
$$
\lambda_{1,p}(M)\int_M|\varphi|^p \leq \int_M|\nabla\varphi|^p,
$$
for any $\varphi \in W^{1,p}_0(M)$.
Given $x\in M$, let us denote by $r(x)$ the distance function on $M$ from a fixed point.
So, given $r>0$, if  we choose
\begin{equation*}
\varphi(x)=\left\{ \begin{array}{ll}
1 & \textrm{ on }B_r,\\
\dfrac{2r-r(x)}{r}  & \textrm{ on } B_{2r}\setminus B_r, \\
0 & \textrm{ on } M\setminus B_{2r},
\end{array}\right.
\end{equation*}
we obtain
\begin{equation}\label{123}
\lambda_{1,p}(M)V(r)\leq r^{-p}V(2r),
\end{equation}
for all $r> 0$.
Assuming, by contradiction, that $\lambda_{1,p}(M)$ is positive and
applying  the volume growth assumption to $V(2r)$ we get
$
V(r)\leq Cr^{k-p},
$
for $r>0$ big enough.

Iterating this argument $\left[\dfrac{k}{p}\right]$ times we obtain
$V(r)\leq C r^a,$
with $a < p$. Now, we use the inequality $(\ref{123})$ to obtain
$$
\lambda_{1,p}(M)V(r)\leq Cr^{a-p}.
$$

Letting $r\rightarrow\infty$, we conclude that $V(M)=0$, which is a contradiction.

\end{proof}


\medskip
%
\section{Proof of Theorem  \ref{ends}}
Given $r>0$, let $B_r$ be a geodesic ball in $M$ centered at some
point $p\in M$. We set $E_r=E\cap B_r$ and $\partial E_r=E\cap\partial B_r$.

Let $f_r$ be the solution of the following Dirichlet problem
$$
\begin{cases}
\De_p f_r=0 \quad \textrm{ in }E_r,\\
f_r=1  \quad \quad \textrm{ in } \partial E, \\
f_r=0 \quad \quad\textrm{ in } \partial E_r.
\end{cases}
$$
By the arguments used  in the proof of Lemma 2.7 in \cite{PST}
$f_r \in C^{1,\alpha}_{loc}(E_r)\cap C(\bar{E}_r)$, $0 < f_r< 1$ in $E_r$, it is increasing
and converges (locally uniformly) to a $p$-harmonic function $f$ with $f \in C^{1,\alpha}_{loc}(E)\cap C(\bar{E})$
satisfying $0 < f \leq 1$ and $f=1$ on $\partial E$.

For a fixed $0<r_0<r$ such that $E_{r_0}\neq \emptyset$, let $\varphi$ be a nonnegative
cut-off function satisfying the properties that
$$
\begin{cases}
\varphi=1 \quad \textrm{ on }E_r\setminus E_{r_0} ,\\
\varphi =0 \quad \quad \textrm{ on } \partial E, \\
|\nabla\varphi| \leq C.
\end{cases}
$$

Applying the inequality ($\ref{sobolev}$) of the assumption and using the fact
that $f_r$ is $p$-harmonic, we obtain

\begin{equation*}
\begin{array}{rll}
\left(\int_{E_r}|\varphi f_r|^p\right)^{p/q}&\leq &C \int_{E_r}|\nabla(\varphi f_r)|^p
= C \int_{E_r}|\varphi\nabla f_r+f_r\nabla \varphi|^p   \\
&\leq & C_1 \int_{E_r}|\varphi\nabla f_r|^p+|f_r\nabla \varphi|^p \\
&\leq&  C_2 \int_{E_r}|f_r|^p|\nabla\varphi|^p \\
&\leq &C_3 \int_{E_r}|f_r|^p,
\end{array}
\end{equation*}
where we have used that $(a+b)^p \leq C(a^p + b^p)$, for a fixed constant $C=2^{p-1}$, and every positive numbers $a$, $b$
in the second inequality, Cacciopoli's Lemma, \ref{cacci}, in the third inequality and  $|\nabla\varphi| \leq C$, in the last inequality.

In particular, for a fixed $r_1$ satisfying $r_0 < r_1 <  r$, we have
$$
\left(\int_{E_{r_1}\setminus E_{r_0}}f_r^q\right)^{p/q} \leq C_3\int_{E_{r_0}}f_r^p.
$$
If $E$ is $p$-parabolic, then the limiting function $f$ is identically $1$. Letting $r \rightarrow \infty$, we obtain
$$
\left(V_E(r_1)-V_E(r_0)\right)^{p/q}\leq C_3\, V_E(r_0),
$$
where $V_E(r)$ denotes the volume of the set $E_r$. Since $r_1 > r_0$ is arbitrary,
this implies that $E$ has finite volume. This conclude proof of the theorem.

\hspace{12cm} $\Box$

\section{Proof of Theorem \ref{nonparabolic}}
Let $f_r$ be the sequence given above and $f$ its limit. 
%
%
%
%
%
%
Let us suppose, by contradiction, that $f\equiv 1$ and $\vol(E)$ is infinite. 
This implies that, given any $L>1$, there exists
$r_1>r_0$ such that $\vol(E_{r_1}-E_{r_0})>2L$.
Since  $f_r\to 1$  uniformly on compact subsets, there exists
$r_2>r_1$ such that $f_r^{\frac{pm}{m-p}}>\frac{1}{2}$ everywhere in 
$E_{r_1}$, for all $r>r_2$. Thus, defining  $h(r):=\int_{E_r-E_{r_0}} f_r^{\frac{pm}{m-p}}$, 
with $r>r_0$, we obtain
\begin{equation}\label{h(r)}
h(r)\geq \int_{E_{r_1}-E_{r_0}} f_r^{\frac{2m}{m-2}}>L,
\end{equation}
for all $r>r_2$. In particular, we have that $\lim_{r\to\infty}h(r)=\infty$.

Now, for each $r>r_0$,  let   
$\varphi=\varphi_r \in C^\infty_0(E)$ be a cut-off function satisfying: 
$$
\begin{cases}
0\leq \varphi \leq 1 \textrm{ everywhere in } E;\\
\varphi \equiv 1 \textrm{ in }  E_r-E_{r_0}.
\end{cases}
$$
By modified Hoffmann-Spruck Inequality \cite{HS} or \cite{BM} we have
\begin{eqnarray*}\label{hsvarphi}
S^{-1}\left(\int_{E_r} (\varphi f_r)^{\frac{pm}{m-p}}\right)^{\frac{m-p}{m}}
\leq \int_{E_r} |\na (\varphi f_r)|^p + \int_{E_r} (\varphi f_r)^p|H|^p,
 \end{eqnarray*}
where $S$ is a positive constant and $p\in (1,m)$.

Using that $f_r\varphi$ vanishes on $\partial E_r$ and the Cacciopoli's Lemma \ref{cacci} 
we obtain
\begin{eqnarray*}\label{hsvarphi}
S^{-1}\left(\int_{E_r} (\varphi f_r)^{\frac{pm}{m-p}}\right)^{\frac{m-p}{m}}
&\leq&C\left( \int_{E_r} f_r^p|\nabla \varphi|^p + \int_{E_r} (\varphi f_r)^p|H|^p\right),
 \end{eqnarray*}
where $C=1+p^p.$ Thus, since $0\leq\varphi\leq 1$ in $E$ and 
$\varphi\equiv 1$ in $E_r-E_{r_0}$, we obtain
\begin{equation}\label{sob-c}
(SC)^{-1}h(r)^{\frac{m-p}{m}}\leq (SC)^{-1}\left(\int_{E_r} (\varphi f_r)^{\frac{pm}{m-p}}
\right)^{\frac{m-p}{m}}
\leq \int_{E_{r_0}} f_r^p|\na \varphi|^p + \int_{E_r} f_r^p|H|^p.
\end{equation}

First, assume that $\|H\|_{L^p(E)}$ is finite. Then, since $0\leq f_r\leq 1$, we have
\begin{equation*}\label{sob-fRp=2}
(SC)^{-1} h(r)^\frac{m-p}{m}
\leq \int_{E_{r_0}}|\na \varphi|^p + \int_{E}|H|^p.
\end{equation*}
Thus, $\lim_{r\to\infty}h(r)<\infty$, which is a contradiction.
Now, assume that $\|H\|_{L^q(E)}$ is finite, for some  $p<q\leq m$. Note that $\frac{m}{m-p}\leq \frac{q}{q-p}$. Since $0\leq f_r\leq 1$ and $h(r)>1$, for all $r>r_2$, we have: 
$$
\begin{cases}
\label{ia} f_r^{\frac{pq}{q-p}}\leq f_r^{\frac{pm}{m-p}}; \\
\label{ib} h(r)^\frac{q-p}{q}\leq h(r)^\frac{m-p}{m}, \textrm{ for all } r>r_2. 
\end{cases}
$$
Thus, using H\"older Inequality, we have
\begin{eqnarray}\label{hder}
\int_{E_r-E_{r_0}} f_r^p |H|^p &\leq& \|H\|^p_{L^q(E_r-E_{r_0})}\left(\int_{E_r-E_{r_0}} f_r^{\frac{pq}{q-p}}\right)^{\frac{q-p}{q}}\\&\leq& \|H\|^p_{L^q(E-E_{r_0})}h(r)^{\frac{m-p}{m}},\nonumber 
\end{eqnarray}
for all $r> r_2$.

Choose $r_0>0$ large enough so that $\|H\|^p_{L^q(E-E_{r_0})}<\frac{1}{2SC}$. Using  (\ref{sob-c}) and (\ref{hder}) we get:
$$ (SC)^{-1} h(r)^{\frac{m-p}{m}}
\leq \int_{E_{r_0}} |\na \varphi|^p+\int_{E_{r_0}}|H|^p + \ \frac{(SC)^{-1}}{2}h(r)^{\frac{m-p}{m}}.$$
This shows that $\lim_{r\to\infty}h(r)<\infty$, which is a contradiction
 and Theorem \ref{nonparabolic} are proved.

\section{Cheng's Theorems for the $p$-Laplacian}\label{applications}

\medskip
Now we describe how we can apply Theorem \ref{ends}  to obtain new Cheng's type inequalities.
For that, we use the Li-Wang approach as in \cite{LW05}.

\medskip

Given a regular domain $\Om\subset M$ let   $\la_1(\Om)$ be the
first Dirichlet eigenvalue of the Laplacian operator.
That is,
$$
\la_1(\Om) = \inf \bigg\{ \frac{\int_\Om |\nabla \varphi|^2}{\int_\Om \varphi^2}:
\varphi \in W^{1,2}_0(\Om)\setminus \{0\} \bigg\}.
$$
We recall that the \emph{bottom of the spectrum} of $M$ is given by
$$
\la_1(M) = \displaystyle\lim_{i\to\infty}\la_1(\Om_i),
$$
where $\{ \Om_i\}_i$ is an exhaustion of $M$,
and this definition does not depend on the exhaustion.

Let $B^{M}_r$ denote a geodesic ball on $M$ with radius $r>0$ and centered at
some point of $M$. The classical Cheng's comparison theorem asserts that, if
$\textrm{Ric}_M \ge -(n-1)$, then $\la_1(B^M_r)\leq \la_1(B^{\mathbb{H}^n}_r)$,
where $\mathbb H^n$ denotes
the $n$-dimensional hyperbolic space $\mathbb H^n$.
One of the consequences is a sharp upper bound for the bottom of the
spectrum on a complete manifold with Ricci curvature bounded from below. Precisely

\smallskip
\begin{maintheorem}(Cheng \cite{Ch})\label{cheng1}
Let $M^m$ be a complete noncompact Riemannian manifold such that
the Ricci curvature of $M$ has a lower bound given by
$$
\textrm{Ric}_M\ge -(m-1).
$$
Then, the bottom of the spectrum of the Laplacian must satisfy the upper
bound
$$
\lambda_1(M)\leq \frac{(m-1)^2}{4}=\la_1(\mathbb H^m).
$$
\end{maintheorem}

\medskip

The Cheng's theorem still holds for  the $p$-Laplacian operator.
An eigenfunction for the Dirichlet problem of the $p$-Laplacian on $\Om\subset M$ is
a nonzero function $u$ such that
$$
\begin{cases}
\Delta_p u= \lambda|u|^{p-2}u \quad \textrm{ in } \Om,\\
u=0 \quad \textrm{ on } \partial\Om,
\end{cases}
$$
for some number $\lambda\in\real$.
We shall denote by  $\la_{1,p}(\Om)$, the smallest eigenvalue of $\De_p$ in $\Om$ for the
Dirichlet problem. It is well known that $\la_{1,p}(\Om)$ has a variational
characterization, analogous to the first eigenvalue of the Laplacian (see  \cite{M})
$$
\la_{1,p}(\Om) = \inf \bigg\{ \frac{\int_\Om |\nabla u|^p}{\int_\Om |u|^p}:
u \in W^{1,p}_0(\Om)\setminus \{0\} \bigg\}.
$$
Using standard comparison ideas, Matei \cite{M} generalized Cheng's result
for the $p$-Laplacian operator, with $p\geq 2$.

\medskip

Using Theorem A and the growth rate of the volume of $2$-hyperbolic ends with positive spectrum
(Theorem 1.4  of  \cite{{LW01}}), Li and Wang proved Cheng's comparison theorem
for K\"ahler manifolds under an assumption on the bisectional curvature.
 Latter, Kong, Li and Zhou \cite{KLZ}
solved the case of quaternionic K\"ahler manifolds.

\medskip

Here we use Theorem \ref{ends} and the volume estimates of
Buckley and Koskela in \cite{BK} to
prove  Cheng's  inequalities for the $p$-Laplacian on K\"ahler and
K\"ahler quaternionic manifolds and thus, we complete the picture for
theses cases. More precisely

\begin{theorem}\label{c1}
Let $M^{2m}$ be a complete noncompact K\"ahler manifold, of real dimension $2m$,
such that the bisectional curvature of $M$ has a lower bound given by
$
\emph{BK}_M\geq -1.
$
Then, for each $p >1$, the bottom of the spectrum of the $p$-Laplacian
must satisfy the upper bound
$$
\lambda_{1,p}(M)\leq \frac{4^p m^p}{p^p}.
$$
Moreover, this estimate is sharp since equality is achieved by the complex hyperbolic space form
 $\mathbb{CH}^{2m}$.
\end{theorem}

\begin{remark} Munteanu (\cite{Mu09}) has obtained a Cheng's comparison theorem
for K\"ahler manifolds under the weaker assumption on Ricci curvature when $p=2$.
However, the techniques we used in this note do not work in that case.
\end{remark}

Following \cite{KLZ}, we are able to obtain Cheng's comparison theorem for quaternionic K\"ahler manifolds,
under a weaker hypothesis on the scalar curvature.

\begin{theorem}\label{c2}
Let $M^{4m}$ be a complete noncompact quaternionic K\"ahler manifold, of real dimension $4m$, such
that the scalar curvature of $M$ has a lower bound given by
$$
\emph{S}_M\geq -16m(m+2).
$$
Then, for each $p >1$, the bottom of the spectrum of the $p$-Laplacian
must satisfy the upper bound
$$
\lambda_{1,p}(M)\leq \frac{2^p (2m+1)^p}{p^p}.
$$
Moreover, this estimate is sharp as equality is a achieved by the quaternionic hyperbolic space form
 $\mathbb{QH}^{4m}$.
\end{theorem}

\medskip

\begin{remark}
We can apply the techniques above to extend the Cheng's comparison theorem of Matei (\cite{M}) for $p>1$.
The Theorems \ref{c1} and \ref{c2} can be obtained by using a $p$-version of Brooks' theorem described
in \cite{LMS} provided the volume of $M$ is infinity.

\end{remark}

Below we provide a unified proof of Theorems \ref{c1} and \ref{c2}.
\medskip

Without loss of generality, we assume that $\lambda_{1,p}(M)$ is positive.
By Theorem \ref{ends} and Lemma \ref{poli} we have that $M$ is $p$-hyperbolic.
Now,  by Theorem 0.1 in \cite{BK} we obtain
$$
V(r) \geq C_0 \exp(p\lambda_{1,p}(M)^{1/p}r),
$$
for all $r \gg 1$ and some $C_0>0$.

\medskip
We point out that our hypotheses on the curvature imply volume growth estimates for
geodesic balls. Namely, $V(r )\leq C\exp(ar)$, where $a = 4m$ in Theorem \ref{c1}
(see \cite {LW05}) and  $a= 2(2m+1)$ in Theorem \ref{c2} (see \cite{KLZ}).

\medskip
Therefore, we get
$$
C_0 \exp(p\lambda_{1,p}(M)^{1/p}r)\leq C\exp(ar),
$$
for all $r\gg1$.
i.e.,
$$
\lambda_{1,p}(M)^{1/p}\leq \dfrac{1}{pr} \ln\left(\dfrac{C}{C_0}\right)+\dfrac{a}{p}.
$$

Letting $r\rightarrow\infty$, we obtain
$$
\lambda_{1,p}(M)\leq \left(\dfrac{a}{p}\right)^p.
$$
In particular we have
\begin{equation}\label{estimateabove}
\lambda_{1,p}(\mathbb{CH}^{2m})\leq \left(\dfrac{4m}{p}\right)^p \qquad\mbox{and}\qquad \lambda_{1,p}(\mathbb{QH}^{4m})\leq \left(\dfrac{2(2m+1)}{p}\right)^p.
\end{equation}

To proof the equality in the space form case we use Theorem 1.1 of
\cite{LMS} applied to the gradient of distance function. Precisely

\medskip
\begin{lemma}[Theorem 1.1 of \cite{LMS}] \label{L3}
Let $\Omega\subset M$ be a domain with compact closure and $\partial \Omega \neq \emptyset$, in 
a Riemannian manifold, $M$. Then
\begin{equation}\label{lms-lema}
\lambda_{1,p}(\Omega)\ge \frac{c(\Omega)^p}{p^p},
\end{equation}
where $c(\Omega)$ is the constant given by
\begin{equation*}
c(\Omega):= \sup\left\{ \frac{\inf_\Omega \mathrm{div} X}{\|X\|_\infty} ;\,\,\, X \in \mathfrak{X}(\Omega)\right\}.
\end{equation*}
Here $\mathfrak{X} (\Omega)$ denotes the set of all smooth vector fields, $X$, on $\Omega$ 
with sup norm $\|X\|_\infty =\sup_\Omega \|X\|<\infty$ (where $\|X\|=g(X,X)^{1/2}$) and $\inf_\Omega \mathrm{div} X>0 $.  
\end{lemma}

\medskip
Now, taking $X=\nabla r$ the gradient of the distance function on $M$, we obtain
$\|X\|=1$ and $\mathrm{div} X= \Delta r$, and consequently $c(\Omega)\ge \inf_{\Omega} \Delta r$.

We point out that, in the space form cases we have

\begin{eqnarray*}
\Delta^{^{\mathbb{CH}}}\, r(x) =2 \coth 2r(x) + 2(2m - 1) \coth r(x)\quad \mbox{on $\mathbb{CH}^{2m}$}
\end{eqnarray*}
and
\begin{eqnarray*}
\Delta^{^{\mathbb{QH}}} r(x) =6 \coth 2r(x) + 4(m - 1) \coth r(x) \quad \mbox{on $\mathbb{QH}^{4m}$}.
\end{eqnarray*}

Thus
\begin{equation*}\label{lambda1estimate}
\inf_\Omega \Delta^{^{\mathbb{CH}}}\, r(x) \ge 4m
\qquad \mbox{and} \qquad 
\inf_\Omega \Delta^{^{\mathbb{QH}}}\, r(x) \ge 2(2m+1)
\end{equation*}
and the result follows from the estimate (\ref{lms-lema}).

\hspace{12cm} $\Box$


\section*{Acknowledgement} The authors wish to thank
Professor S. Pigola for helpful comments about this paper.


\vspace{1cm}

\end{document}